\documentclass[11pt,a4paper]{amsart}

\usepackage{amsmath,amssymb,amsfonts}
\newcommand{\PMP}{Pontryagin's Maximum Principle}
\newcommand{\R}[1]{\mathbb{R}^{#1}}
\newcommand{\N}[1]{\mathbb{N}^{#1}}
\newcommand{\U}[1]{\mathcal{U}}
\newcommand{\set}[1]{\ensuremath{\mathcal{#1}}}
\newcommand{\spanby}[1]{\ensuremath{\operatorname{span} \left\{ {#1} \right\}}}
\newcommand{\innerprod}[2]{\ensuremath{\left\langle {#1} , {#2} \right\rangle}}
\newcommand{\lieprod}[2]{\ensuremath{\left[ {#1} , {#2} \right]}}
\newcommand{\ad}[3]{\ensuremath{\operatorname{ad}_{#1}^{#2}{#3}}}
\newcommand{\sign}[1]{\ensuremath{\operatorname{sign} \left( {#1} \right)}}
\newcommand{\mes}[1]{\ensuremath{\operatorname{Mes} \left( {#1} \right)}}

\usepackage{amsthm}
\newtheorem{teo}{Theorem}[section]
\newtheorem{lem}[teo]{Lemma}

\newtheorem{cor}[teo]{Corolary}
\newtheorem{defin}[teo]{Definition}

\newtheorem{remark}[teo]{Remark}

\usepackage{setspace}
\begin{document}
\title{Fuller Phenomenon in multiple input control systems}
\author{Eduardo Oda}
\address{IMEUSP - Instituto de Matem\'{a}tica e Estat\'{\i}stica da Universidade de S\~{a}o Paulo}
\email{oda@ime.usp.br}
\author{Pedro Aladar Tonelli}
\address{IMEUSP - Instituto de Matem\'{a}tica e Estat\'{\i}stica da Universidade de S\~{a}o Paulo}
\email{tonelli@ime.usp.br}

\begin{abstract}
  Many optimal control problems exhibit a peculiar behavior that is not completely understood, the Fuller Phenomenon. In a naive way, this phenomenon can be described as the accumulation of discontinuities in the control function.

  In this paper extensions to multiple input control systems of classic results on the detection of this behavior are given. It is also given a necessary condition to an arc be singular. This condition gives a potentially new direction of $p^\bot$ which is used to extend the First Pontryagin Cone, improving the geometric comprehension of the problem.

  These techniques are applied to control systems derived from Hamiltonian systems, and sufficient conditions for existence of the Fuller Phenomenon in a subfamily of these systems are given.
\end{abstract}

\maketitle

A common matter on control systems is the optimization of a given criteria. The search for optimal trajectories are mainly supported by the Pontryagin Maximum Principle which often leads to a discontinuous control. Actually it is the general behavior of affine systems.

However, in 1963 Anthony T. Fuller \cite{FULLER:1963} has presented an optimal control problem with accumulation of discontinuities. Although it were unexpected at that time, Ivan A. K. Kupka, in 1990 \cite{KUPKA:1990}, proved that this is not an exceptional behavior.

Due to Kupka, this accumulation of discontinuities is known as {\em Fuller Phenomenon}. This phenomenon is challenging in understand the optimal trajectories and it has been studied by several authors, but the are still many questions to be answered. A particular and interesting case is the junction of singular and non singular arcs of multi-input systems.

In this paper extensions to multiple input control systems of classic results on the detection of this behavior are given. It is also given a necessary condition to an arc be singular. This condition gives a potentially new direction of $p^\bot$ which is used to extend the First Pontryagin Cone, improving the geometric comprehension of the problem. These techniques are applied to control systems derived from Hamiltonian systems, and sufficient conditions for existence of the Fuller Phenomenon in a subfamily of these systems are given.

%\section[Introduction]{Introduction\protect\footnote{This section is very much the same the introduction of another paper of the same authors \cite{ODA-TONELLI:2013}, but in seek of a self contained paper, they have chosen to reproduce it here.}}
\section{Introduction}
  Consider the following optimal affine control problem (P):
  \begin{align*}
    \mbox{minimize } & \int_{0}^{T_f(u)} f_0(x)+\sum_{i=0}^{m}g_{0i}(x)u_i \; dt \\
    \mbox{subject to } & \begin{cases}
        \dot{x} = f(x)+\sum_{i=0}^{m}g_i(x)u_i \\
        u=(u_1,\dots,u_m):[0,T_f(u)]\rightarrow\R{m} \mbox{ such that }\\
        |u_i(t)|\leq K(t),\; \forall t\in [0,T_f(u)],\; i=1, \dots, m \\
        x(0) = A \\
        x(T_f(u)) = B
      \end{cases}
  \end{align*}
  where:
  \begin{enumerate}
    \item $x=(x_1,\dots,x_n)\in\R{n}$ and $u=(u_1,\dots,u_m)\in\R{m}$
    \item $f$, $g_i$, $i=1, \dots, m$, are analytic vector fields in $\R{n}$
    \item $f_0$, $g_{0i}$, $i=1, \dots, m$, are analytic maps from $\R{n}$ to $\R{}$
    \item $K$ is analytic and strictly positive
    \item $u_i \in \mes{\R{}}$, $i=1, \dots, m$
  \end{enumerate}
  
  From \PMP{}, if one defines the Hamiltonian function:
  \begin{align*}
    H_\lambda:T^*\R{n} \times \R{m} &\longrightarrow \R{} \\
    (x,p,u) &\longmapsto \innerprod{p}{f(x)+\sum_{i=0}^{m}g_i(x)u_i} - \lambda \left( f_0(x)+\sum_{i=0}^{m}g_{0i}(x)u_i \right)
  \end{align*}
  where $(x,p) \in T_x\R{n}$ and $\lambda \in \{0,1\}$, then each optimal trajectory $(\bar{x},\bar{u}):[0,\bar{T}] \rightarrow \R{n} \times \R{m}$ has a lift to the cotangent space such that $H_\lambda(\bar{x},\bar{p},\bar{u})=0$ and
  \begin{equation*}
    \mbox{(Adj) }
      \begin{cases}
      \frac{d\bar{x}}{dt}(t)=\frac{\partial H_\lambda}{\partial p}(\bar{x}(t),\bar{p}(t),\bar{u}(t)) \\
      \frac{d\bar{p}}{dt}(t)=-\frac{\partial H_\lambda}{\partial x}(\bar{x}(t),\bar{p}(t),\bar{u}(t)) \\
      H_\lambda(\bar{x}(t),\bar{p}(t),\bar{u}(t))=\sup\left\{ H_\lambda(\bar{x}(t),\bar{p}(t),v) | v \in U \right\}
      \end{cases}       
  \end{equation*} 
  for almost all $t \in [0,\bar{T}]$ and $\lambda \in \{0,1\}$. Also $(\lambda,\bar{p}(t))\neq 0$, for almost all $t\in [0,\bar{T}]$. The solutions of (Adj) are called extremals and might not be an optimal solution of the original problem. The new variable $p$ is known as the \emph{adjoint variable}.

  Setting
  \begin{align*}
    \bar{f}&=(f_0,f) & \bar{g}_i&=(g_{0i},g_i) \\
    \bar{x}&=(x_0,x) & \bar{p}&=(\lambda,p)
  \end{align*}
  where $x_0$ satisfies:
  \begin{equation*}
    \dot{x}_0=f_0(x)+\sum_{i=0}^{m}g_{0i}(x)u_i,
  \end{equation*}
  the Hamiltonian function become $H_\lambda=\innerprod{\bar{p}}{\bar{f}} + \sum_{i=i}^{m} \innerprod{\bar{p}}{\bar{g}_i}u_i$ and one has:
  \begin{equation} \label{e:adj_simples}
    \begin{aligned}
      \dot{\bar{x}}&=\frac{\partial H_\lambda}{\partial \bar{p}} = \bar{f}+\sum_{i=0}^{m}\bar{g}_i(x)u_i\\
      \dot{\bar{p}}&=-\frac{\partial H_\lambda}{\partial \bar{x}} = -\bar{p}\frac{\partial f}{\partial x} - \sum_{i=1}^{m} \bar{p}\frac{\partial g_i}{\partial x}u_i.
    \end{aligned}
  \end{equation}
  
  Henceforth, the simplified notation $f, g_i, x$ and $p$ stands for $\bar{f}, \bar{g}_i, \bar{x}$, respectively. Note that the dimension of the new problem is $n+1$.
  
  Since the Hamiltonian $H_\lambda$ is linear in $u$, one has from \PMP{} that $u_i(t)=\sign{\innerprod{p(t)}{g_i(x(t))}} K(t)$ on the intervals where $\innerprod{p(t)}{g_i(x(t))}$ is non zero almost everywhere. These are known as the \emph{nonsingular intervals}. In this case the control is said nonsingular on these intervals and the trajectory is called nonsingular arc on these intervals.
    
  Analogously, the intervals where $\innerprod{p(t)}{g_i(x(t))}$ vanishes almost everywhere are known as \emph{singular intervals} and the controls are called singular on these intervals and the trajectory is called singular arc on these intervals. Along a singular interval the control can not be designed by \PMP{}, but one could consider the time derivatives of $\innerprod{p(t)}{g_i(x(t))}$.

  Indeed, in a singular interval, the time derivatives of the column vector $\left[\innerprod{p}{g_i}\right]_{i=1,\dots,m}$ could be evaluated until a relation that depends on $u$ explicitly is obtained. In other words, for $l \in \N{}$,  one has the $m \times m$ matrix:
  \begin{equation*}
    B_l = \frac{\partial}{\partial u} \left( \frac{d^l}{dt^l} \left[\innerprod{p}{g_i}\right]_{i=1,\dots,m} \right).
  \end{equation*}
  
  Consider the first one that is not identically zero, say the $k$-th derivative $B_k$. Then one has:
  \begin{equation*}
    0 = \frac{d^k}{dt^k} \left[\innerprod{p}{g_i}\right]_{i=1,\dots,m} = A_k(x,p) + B_k(x,p)u.
  \end{equation*}
  
  If $B_{k}$ is nonsingular, all the controls can be evaluated. Otherwise, it is necessary to find some controls, reducing the problem and restarting the procedure. It's now clear the central role of the functions $\innerprod{p(t)}{g_i(x(t))}$ and its derivatives on the design of optimal controls.
  
  The number $q=k/2$ is called \emph{problem order} or \emph{intrinsic order}. Note that even if $B_k$ is not identically zero it can became singular, or even identically zero, along an specific extremal. This fact lead us to another concept of order, know as \emph{arc order} or \emph{local order}. This another concept will not be used in this work and it is better discussed on \cite{LEWIS:1980} and \cite{ODA:2008}.
  
  In 1967 Robbins \cite{ROBBINS:1967} has proved that the intrinsic order is a integer (i.e., $k$ is even) and that the matrix $(-1)^{q}B_{2q}$ is negative semidefinite. This is a extension of Legendre-Clebsch criteria and know as Generalized Legendre-Clebsch Condition (GLC) \footnote{
    Some authors, like \cite{LEWIS:1980}, say that GLC Condition holds for the problem order, but in \cite{ODA-TONELLI:2013} a counterexample is given for the case of multiple input systems and it is shown that the problem order is a positive integer if the system has only one input.
  }. If the matrix is definite we say that it satisfies the Generalized Legendre-Clebsch Strict Condition (GLCS).
  
  To evaluate the problem order a simple known lemma is necessary. It was obtained independently by several authors and its proof is rather simple (see \cite{ODA-TONELLI:2013}).
  
  \begin{lem} \label{l:derivada_de_<p,h>}
    Let $h$ be a smooth vector field. Then along an extremal we have:
    \begin{equation*}
      \frac{d}{dt}\innerprod{p}{h}=\innerprod{p}{\lieprod{f}{h}+\sum_{i=1}^{m}u_i\lieprod{g_i}{h}}.
    \end{equation*}
  \end{lem}

\section{Some useful definitions}
  In this paper the points of interest are those at the border of singular and nonsingular intervals. These time instants are called junctions of singular and nonsingular arcs.

  There are more three definitions that will be used in the next section:
  
  \begin{defin} [Piecewise analytic]
    A real valued function $f:\R{}\rightarrow\R{}$ is piecewise analytic on an interval $(a,b)$ if for each $t_c\in (a,b)$ there exist $t_1\in(a,t_c)$ and $t_2\in(t_c,b)$ shush that $f$ is analytic on $(t_1,t_c)$ and $(t_c,t_2)$. A map $g:\R{}\rightarrow\R{n}$ is 
    piecewise analytic on $(a,b)$ if its components are piecewise analytic.
  \end{defin}
  
  \begin{defin} [Analytic junction]
    A junction is said to be analytic if the optimal control is piecewise analytic in a neighborhood of the junction.
  \end{defin}

  \begin{defin} [Analytic junction of order $q$] \label{d:juncao_analitiva_ordem_q}
    Given a $q$-order problem and a trajectory with a junction at $t_c$, this junction is said to be analytic of order $q$ if it is analytic and the matrix $B_{2q}$ in nonsingular in a neighborhood of the junction.
  \end{defin}

  For the following sections we will denote $A_{2q}$ e $B_{2q}$ by $A$ and $B$, respectively.
  
\section{The junction order parity in multiple input control systems}
  In 1971, McDanell and Powers \cite{MCDANELL-POWERS:1971} have proved a theorem relating the junction order with the control derivatives. This theorem will be extended to the multiple input case.
  
  To avoid the confusion about $B$ been null along the extremal, it will be supposed that $B$ is nonsingular along the trajectory (see definition \ref{d:juncao_analitiva_ordem_q}).
    
  \begin{teo} \label{t:q+r impar}
    Let $t_c$ be an analytic junction of order $q$ in a singular arc such that $\lVert u_i \rVert < K(t)$, $i=1,\dots,m$. If $u^{(r)}$, $r \geq 0$, is the lowest order derivative of the control which is discontinuous at $t_c$, then $q+r$ is odd.
  \end{teo}

  \begin{proof}
    Consider the function $\phi(t)=\left( \phi_1(t), \dots, \phi_m(t) \right)$, where
    \begin{equation*}
      \phi_i(t)=\innerprod{p(t)}{g_i(x(t))}     
    \end{equation*}
    with $i=1,\dots,m$. Clearly, $\phi$ is a $\set{C}^{2q+r-1}$ function at $t_c$ and $\phi$ do not vanish in a neighborhood of $t_c$ intercepted by the interior of the non singular arc.
      
    There is an $\epsilon \neq 0$ such that $t_c+\epsilon$ is in the intersection of this neighborhood and the interior of the non singular arc, $t_c-\epsilon$ is in the singular arc and $B$ is negative definite on $(t_c-\epsilon,t_c+\epsilon)$, i.e., we have GLCS in this interval. From now on we will consider this neighborhood.
      
    Let's denote by $u_s$ and $u_n$ the restrictions of $u$ to the singular and non singular arcs, respectively, and
    \begin{align*} 
      u_s^{(i)}(t_c) &= \lim_{\epsilon \rightarrow 0} u^{(i)}(t_c - \epsilon) \\
      u_n^{(i)}(t_c) &= \lim_{\epsilon \rightarrow 0} u^{(i)}(t_c + \epsilon).
    \end{align*}

    Denoting $2q+r$ by $k$, since $\phi^{(i)}$ is continuous at $t_c$ for all $0 \leq i \leq k-1$ and $\phi\equiv 0$ along the singular interval, the first non null portion of the Taylor formulae of $\phi$ around $t_c$ is the one related to $\phi^{(k)}$:
    \begin{equation*} 
      \phi(t_c+\epsilon)=\frac{\epsilon^k}{k!}\phi^{(k)}(t_c) + o(\epsilon^{k}).
    \end{equation*}
      
    But note that:
    \begin{equation*} 
      \phi^{(k)}=\frac{d^r\phi^{2q}}{dt^r}=\frac{d^r}{dt^r}\left( A + B u \right),
    \end{equation*}
    then:
    \begin{equation} \label{e:fuller_taylor_1}
      \phi(t_c+\epsilon)=\frac{\epsilon^k}{k!}\left( A^{(r)}(t_c) +\sum_{i=0}^{r} \binom{r}{i} B^{(r-i)}(t_c)u_n^{(i)}(t_c) \right) + o(\epsilon^{k}).
    \end{equation}
      
    Along the singular interval:
    \begin{equation*} 
      0=\phi^{(2q)}=A+B u_s \Rightarrow A=-B u_s,
    \end{equation*}
    therefore
    \begin{equation*} 
      A^{(r)}(t_c)=-\sum_{i=0}^{r} \binom{r}{i} B^{(r-i)}(t_c)u_s^{(i)}.
    \end{equation*}

    Using this identity in the equation \eqref{e:fuller_taylor_1}:
    \begin{equation*}
      \phi(t_c+\epsilon)=\frac{\epsilon^k}{k!}\left( \sum_{i=0}^{r} \binom{r}{i} B^{(r-i)}(t_c)\left(u_n^{(i)}(t_c) - u_s^{(i)}(t_c) \right) \right) + o(\epsilon^{k}).
    \end{equation*}

    Since $u_n^{(i)}(t_c)=u_s^{(i)}(t_c)$ for all $0 \leq i \leq r-1$, one finally gets:
    \begin{equation} \label{e:fuller_taylor_2}
      \phi(t_c+\epsilon)=\frac{\epsilon^k}{k!}B(t_c)\left(u_n^{(r)}(t_c) - u_s^{(r)}(t_c)\right) + o(\epsilon^{k}).
    \end{equation}
      
    Consider the vector $\sigma=\left( \sign{\phi_1(t_c+\epsilon)} , \dots , \sign{\phi_m(t_c+\epsilon)} \right)$. Since the junction is analytic, $\sigma$ is constant in a neighborhood of the junction intercepted by the nonsingular interval. And it is known from \PMP{} that on the nonsingular interval $u_n(t)=\sigma K(t)$, therefore $u_n^{(i)}(t_c)=\sigma K^{(i)}(t_c)$, $i=0,\ldots,r$. Consider now the following expansion:
    \begin{equation} \label{e:fuller_taylor_3}
      \begin{split}
        \sigma K(t_c-\epsilon) - u(t_c-\epsilon) &= \sum_{i=0}^{r}\frac{\left(-\epsilon\right)^i}{i!} \left( \sigma K^{(i)}(t_c)-u_s^{(i)}(t_c) \right) + o(\epsilon^{r}) \\
        &= \sum_{i=0}^{r}\frac{\left(-\epsilon\right)^i}{i!} \left( u_n^{(i)}(t_c)-u_s^{(i)}(t_c) \right) + o(\epsilon^{r}) \\
        &= \frac{\left(-1\right)^r\epsilon^r}{r!} \left( u_n^{(r)}(t_c)-u_s^{(r)}(t_c) \right) + o(\epsilon^{r}).
      \end{split}
    \end{equation}
    from where one gets:
    \begin{equation*} 
      u_n^{(r)}(t_c)-u_s^{(r)}(t_c) = \frac{\left(-1\right)^r r!}{\epsilon^r} \left( \sigma K(t_c-\epsilon) - u(t_c-\epsilon) \right) + o(\epsilon^{r}).
    \end{equation*}

    Using this relation in the equation \eqref{e:fuller_taylor_2}, one has:
    \begin{equation*}
      \phi(t_c+\epsilon)=\frac{\left(-1\right)^r r! \epsilon^{2q}}{k!}B(t_c)\left( \sigma K(t_c-\epsilon) - u(t_c-\epsilon) \right) + o(\epsilon^{k}).
    \end{equation*}

    Since $\sigma$ was defined in the way that $\innerprod{\phi(t_c+\epsilon)}{\sigma K(t_c-\epsilon) - u(t_c-\epsilon)}$ is positive and denoting $v=\sigma K(t_c-\epsilon) - u(t_c-\epsilon)$, one has:
    \begin{equation*}
      0 < \innerprod{v}{\phi(t_c+\epsilon)} = \innerprod{v}{\frac{\left(-1\right)^r r! \epsilon^{2q}}{k!}B(t_c)v + o(\epsilon^{k})},
    \end{equation*}
    from what:
    \begin{equation} \label{e:fuller_inequacao_1}
      0 < (-1)^r \innerprod{v}{B(t_c)v}.
    \end{equation}

    The hypothesis $\lVert u_i \rVert < K(t)$, $i=1,\dots,m$, implies that each component of $v$ is non zero. Then, since $(-1)^qB(t_c)$ is negative definite, one can conclude that $\innerprod{v}{(-1)^qB(t_c)v} < 0$. Thus, by multiplying both sides of the inequation \eqref{e:fuller_inequacao_1} by this quantity, one gets:
    \begin{equation*}
      0 > (-1)^r \innerprod{v}{B(t_c)v}.\innerprod{v}{(-1)^qB(t_c)v} = (-1)^{q+r} \innerprod{v}{B(t_c)v}^2
    \end{equation*}  
    Therefore one can finally concludes that $(-1)^{q+r}<0$. Therefore, $q+r$ is odd.
  \end{proof}
    
  \begin{cor} \label{c:fuller_u_nao_analitico_por_partes}
    With the same hypothesis of theorem, if $q$ is even and $A(t_c)+K(t_c)B(t_c)v \neq 0$ for all $v \in \left\{ -1,1 \right\}^m$, then the junction is not analytic.
  \end{cor}
  \begin{proof} 
    It is enough to show that the control is discontinuous, i.e., $r=0$. To do so, note that there is $v \in \left\{ -1,1 \right\}^m$ such that the control on the nonsingular interval is given by $u(t)=vK(t)$. Thus, for this specific $v$, $A(t_c)+B(t_c)u_n(t_c)=A(t_c)+K(t_c)B(t_c)v \neq 0=A(t_c)+B(t_c)u_s(t_c)$. Therefore the control is discontinuous.
  \end{proof}

  \begin{cor} \label{c:fuller_u_nao_analitico_por_partes alpha=0}
    With the same hypothesis of theorem, if $q$ is even and $A \equiv 0$, then the junction is not analytic.
  \end{cor}
  \begin{proof} 
    Indeed, in this case the control is null on the singular interval and at the border of the admissible set on the nonsingular interval, therefore it is discontinuous.
  \end{proof}

  \section{An extension of Fuller's example}
    In 1963, Fuller propose the following example:
    \begin{align*}
      \mbox{minimize } & \int_{0}^{T_f(u)} \frac{x^2}{2} \; dt \\
      \mbox{subject to } & \begin{cases}
          \dot{x} = v \\
          \dot{v} = u \\
          u\in\mes{\R{}}, & |u(t)|\leq 1,\; \forall t\in [0,T_f(u)]\\
          (x,v)(0) = A \\ (x,v)(T_f(u)) = B\\
        \end{cases}
    \end{align*}
    
    The control system of this problem can be obviously interpreted as the dynamics of a point along a straight line with limited acceleration. In the Hamiltonian formalism it can be described with the function $H(x,v) = T(x,v) + P(x)$, where:
    \begin{align*}
      T(x,v)&=\frac{v^2}{2} & P(x)&=-xu \\
    \end{align*}
    and the functional to be optimized is the square of the norm of $x$.
    
    Thereby, the Fuller's example can be extended to a multiple input system with the following generalizations of the kinetic energy and of the potential energy:
    \begin{align*}
      T(x,v)&=\frac{v^\intercal M_1 v}{2} & P(x)&=-u^\intercal M_2 x
    \end{align*}
    where $(x,v)\in\R{n}\times\R{n}$, $u \in \R{n}$,  $M_1$ and $M_2$ are $n\times n$ matrices which are symmetric and constants, $M_1$ is positive definite, $M_2$ is invertible and the functional to be optimized is $\frac{\lVert x \rVert^2}{2}$. The new problem is:
    \begin{align*}
      \mbox{minimize } & \int_{0}^{T_f(u)} \frac{\lVert x \rVert^2}{2} \; dt \\
      \mbox{subject to } & \begin{cases}
	  \dot{x} = M_1 v \\
	  \dot{v} = M_2 u \\
	  u\in\mes{\R{n}}, \\
	  |u_i(t)|\leq 1,& \forall t\in [0,T_f(u)], \;i=1,\dots,n \\
	  (x,v)(0) = A \\ (x,v)(T_f(u)) = B\\
	\end{cases}
    \end{align*}
        
    The Hamiltonian of \PMP{} for this new problem is:
    \begin{equation*}
      H_\lambda(x,p,u)= p_1^\intercal M_1 v + p_2^\intercal M_2 u - \lambda \frac{\lVert x \rVert^2}{2}
    \end{equation*}
    where $(p_1,p_2)\in\R{n}\times\R{n}$. Therefore, $u=\sign{M_2 p_2}$ on the nonsingular intervals and the adjoint equations are:
    \begin{equation*}
      \mbox{(Adj) }
	\begin{cases}
	  \dot{x}=M_1 v &
	  \dot{v}=M_2 u \\
	  \dot{p}_1=x &
	  \dot{p}_2=-M_1 p_1.
	\end{cases}
    \end{equation*}
    
    Note that, standing at the origin, the optimal control is $u\equiv 0$. Therefore, any optimal trajectory that reaches the origin can be extended indefinitely with a singular arc. On the other hand, along a singular interval, $\phi=M_2 p_2\equiv0$, thus:   
    \begin{align*}
      0=\phi^{(1)} &= M_2 \dot{p}_2=-M_2 M_1p_1 \\
      0=\phi^{(2)} &= -M_2 M_1\dot{p}_1=-M_2 M_1 x \\
      0=\phi^{(3)} &= -M_2 M_1\dot{x}=-M_2 M_1 M_1 v \\
      0=\phi^{(4)} &= -M_2 M_1 M_1\dot{v}=-M_2 M_1 M_1 M_2 u.
    \end{align*}
    
    Thus $A=0$ and $B=-M_2 M_1 M_1 M_2$. Note that $B$ is invertible and the problem (and each arc, since it is constant) has order $q=2$. Then, from the corollary \ref{c:fuller_u_nao_analitico_por_partes alpha=0}, one can concludes that the junction can not be analytic. Moreover, note that $(-1)^2B$ is negative definite, which is in accordance with the GLC Condition.
    
    One might notice that $A$ and $B$ been constant was crucial to prove that the control was discontinuous and that the order (both of them) is 2. In a more general situation, this could be much more difficult. In the next session some directions on this matter are given.
    
  \section{Revisiting the Pontryagin Cone}
    In this section the lemma \ref{l:derivada_de_<p,h>} will be used to explicit $A$ and $B$. This evaluations will let clear the role of the distribution generated by vector fields that define the problem.
    
    It will be considered the system with the augmented dimension like the equations \ref{e:adj_simples}, with $H_\lambda=\innerprod{p}{f}+\sum_{i=1}^{m}\innerprod{p}{g_i}u_i$ e $\phi=(\innerprod{p}{g_1},\dots,\innerprod{p}{g_m})$.
    
    The first derivative of $\phi$ has the form:
    \begin{equation*}
      \phi^{(1)} =
        \left[\begin{array}{c}
          \frac{d}{dt}\innerprod{p}{g_1} \\
          \vdots \\
          \frac{d}{dt}\innerprod{p}{g_m} 
        \end{array}\right]
        =
        \left[\begin{array}{c}
          \innerprod{p}{ \ad{f}{}{g_1} + \sum_{i=1}^{m} \lieprod{g_i}{g_1}u_i } \\
          \vdots \\
          \innerprod{p}{ \ad{f}{}{g_m} + \sum_{i=1}^{m} \lieprod{g_i}{g_m}u_i }
        \end{array}\right].
    \end{equation*}
    
    If the first derivative do not depend on $u$, in other words, if $\lieprod{g_i}{g_j}=0$, $\forall i,j$, then the second derivative can be evaluated:
    \begin{equation*}
      \phi^{(2)} =
	\left[\begin{array}{c}
	  \frac{d}{dt}\innerprod{p}{ \ad{f}{}{g_1} } \\
	  \vdots \\
	  \frac{d}{dt}\innerprod{p}{ \ad{f}{}{g_m} }
	\end{array}\right]
	=
	\left[\begin{array}{c}
	  \innerprod{p}{ \ad{f}{2}{g_1} + \sum_{i=1}^{m} \lieprod{g_i}{\ad{f}{}{g_1}}u_i } \\
	  \vdots \\
	  \innerprod{p}{ \ad{f}{2}{g_m} + \sum_{i=1}^{m} \lieprod{g_i}{\ad{f}{}{g_m}}u_i }
	\end{array}\right].
    \end{equation*}
    
    Proceeding in this way until one finds an expression which explicitly depends on $u$ one gets:
    \begin{equation*}
      \phi^{(k)} =
	\left[\begin{array}{c}
	  \innerprod{p}{ \ad{f}{k}{g_1} + \sum_{i=1}^{m} \lieprod{g_i}{\ad{f}{k-1}{g_1}}u_i } \\
	  \vdots \\
	  \innerprod{p}{ \ad{f}{k}{g_m} + \sum_{i=1}^{m} \lieprod{g_i}{\ad{f}{k-1}{g_m}}u_i }
	\end{array}\right].
    \end{equation*}
    
    It is possible to rewrite this expression in the form $\phi^{(k)} = A_k + B_ku$ where:
    \begin{equation*}
      A=\left[\begin{array}{c}
	\innerprod{p}{\ad{f}{k}{g_1}} \\
	\innerprod{p}{\ad{f}{k}{g_2}} \\
	\vdots \\
	\innerprod{p}{\ad{f}{k}{g_m}}
      \end{array}\right]
    \end{equation*}
    \begin{equation*}
      B=\left[\begin{array}{cccc}
	\innerprod{p}{\lieprod{g_1}{\ad{f}{k-1}{g_1}}} & \innerprod{p}{\lieprod{g_2}{\ad{f}{k-1}{g_1}}} & \cdots & \innerprod{p}{\lieprod{g_m}{\ad{f}{k-1}{g_1}}} \\
	\innerprod{p}{\lieprod{g_1}{\ad{f}{k-1}{g_2}}} & \innerprod{p}{\lieprod{g_2}{\ad{f}{k-1}{g_2}}} & \cdots & \innerprod{p}{\lieprod{g_m}{\ad{f}{k-1}{g_2}}} \\
	\vdots & \vdots & \ddots & \vdots \\
	\innerprod{p}{\lieprod{g_1}{\ad{f}{k-1}{g_m}}} & \innerprod{p}{\lieprod{g_2}{\ad{f}{k-1}{g_m}}} & \cdots & \innerprod{p}{\lieprod{g_m}{\ad{f}{k-1}{g_m}}}
      \end{array}\right].
    \end{equation*}
    
    Somewhat hidden in these evaluations is the fact that $\innerprod{p}{\ad{f}{l}{g_i}}=0$ for all $i=1,\dots,m$, $j=1,\dots,k-1$ and for all $t$ in the singular interval, because there are the entries of the vector $A$ . Besides, since it is a singular arc, $\innerprod{p}{g_i}=0$, then it is clear that the distribuition
    \begin{equation*}
      \spanby{\ad{f}{l}{g_i}\;|\; i=1,\dots,m;\; l=0,\dots,k-1}
    \end{equation*}
    is orthogonal to $p$. This is a well known distribuition called the First Pontryagin Cone and it is very important because it gives information about the adjoint variable $p$. But along a singular arc we can give another direction of $p^\bot$.
    
    \begin{lem}
      Along a singular arc $\innerprod{p}{f}=0$.
    \end{lem}
     
    \begin{proof}
      We know from \PMP{} that along an optimal solution $H_\lambda = 0$. On the other hand, along a singular arc, $H_\lambda = \innerprod{p}{f}$, therefore $\innerprod{p}{f}=0$.
    \end{proof}

    \begin{remark}
      Note that is a necessary condition to an arc be singular.
    \end{remark}

    Since it is a subspace of $p^\bot$, the dimension of the distribution
    \begin{equation*}
      \Delta = \spanby{f,\;\ad{f}{l}{g_i}\;|\; i=1,\dots,m;\; l=0,\dots,k-1} 
    \end{equation*}
    tells a lot about the complexity of the problem. Indeed, remember that $p\in\R{n+1}$ and that $\Delta \subset p^\bot$. Then, if $\dim{\Delta}=n+1$, one has $p\equiv0$, which contradicts \PMP{}. On the hand, if $\dim{\Delta}=n$, then $p^\bot$ was fully described and, then, the direction of $p$ is known. This essentially gives all the information needed to solve the optimal control problem. If $\dim{\Delta}<n$, then it is hard to solve problem due to missing information, represented by the unknown directions of $p^\bot$. In some sense, $n-\dim{\Delta}$ tells how much it is known about the problem.
    
    This analysis can also help to determinate when the matrix $B$ is invertible, which is a fundamental issue in the previous section. To this matter, consider the following definition.
    
    \begin{defin}
      Let $M$ be a matrix of the form:
      \begin{equation*}
        M=\left[\begin{array}{cccc}
          \innerprod{p}{v_{11}} & \innerprod{p}{v_{12}} & \cdots & \innerprod{p}{v_{1l}} \\
          \innerprod{p}{v_{21}} & \innerprod{p}{v_{22}} & \cdots & \innerprod{p}{v_{2l}} \\
          \vdots & \vdots & \ddots & \vdots \\
          \innerprod{p}{v_{l1}} & \innerprod{p}{v_{l2}} & \cdots & \innerprod{p}{v_{ll}}
        \end{array}\right].
      \end{equation*}
      
      If the constant entries of this matrix and the vector fields $v_{ij}$ that belongs to $\Delta$ (and therefore result in null entry in this matrix) ensure that $M$ is invertible, then the matrix $M$ is said to be \emph{$\Delta$-inverse-decidable}.
    \end{defin}

    This definition motivates the following theorem.
    
    \begin{teo}
      If a junction has order $q$ even, $\ad{f}{2q}{g_i}\in\Delta$ and $B$ is $\Delta$-inverse-decidable, then there is Fuller Phenomenon at this junction.
    \end{teo}
    \begin{proof}
      Considering the previous discussion the proof is rather trivial. Just note that the hypothesis of the corollary \ref{c:fuller_u_nao_analitico_por_partes alpha=0} are verified.
    \end{proof}

  \section{Fuller Phenomenon in Hamiltonian systems}
    Many control problems have their origin at mechanical systems, usually obtained by the introduction of external forces. These forces are the controls and usually they have physical restrictions. Thus, if one wants the system to optimally perform some task, then it became an optimal control problem.

    The Fuller's example and its extension are of this kind of problem. It rises the question of how general is the Fuller Phenomenon in mechanical systems. In this section, it will be shown that there is a family of Hamiltonian systems that has accumulation of discontinuities. At the end, other families are also discussed.
    
    Consider the Hamiltonian function:
    \begin{equation*}
      \mathcal{H}=\frac{v^\intercal T v}{2} + Q(x) - x^\intercal Mu
    \end{equation*}
    where $(x,v)\in \R{n}\times\R{n}$ and $u \in \R{n}$. Suppose that, under this dynamic, starting the trajectory at $(x_0,v_0)\in\R{n}\times\R{n}$ and reaching the origin, the functional $c:x\in\R{n}\rightarrow\R{}$ needs to be minimized.
    
    \begin{align*}
      \mbox{minimize } & \int_{0}^{T_f(u)} c(x) \; dt \\
      \mbox{subject to } & \begin{cases}
	  \dot{x} = T v \\
	  \dot{v} = P(x) + Mu\\
	  u\in\mes{\R{n}}, \\
	  |u_i(t)|\leq 1,& \forall t\in [0,T_f(u)], \;i=1,\dots,n \\
	  (x,v)(0) = (x_0,v_0) \\ (x,v)(T_f(u)) = 0\\
	\end{cases}
    \end{align*}
    where $P=-\frac{\partial Q}{\partial x}$. It will be supposed that:
    \begin{enumerate}
      \item $T$ is symmetric positive definite;
      \item $M$ is symmetric and invertible;
      \item $Q$ and $c$ are $\mathcal{C}^\infty$ maps;
      \item $P$ vanishes at the origin;
      \item $c(x)=0 \Leftrightarrow x=0$;
      \item $\frac{\partial c}{\partial x}(0)=0$ and $\frac{\partial^2 c}{\partial x^2}$ is positive definite.
    \end{enumerate}

    Since the origin is an equilibrium, as soon as the trajectory reaches the origin all the controls vanish. Also, any trajectory that reaches the origin can be indefinitely extended with a singular arc, without increasing its cost.
    
    It will be shown that an optimal trajectory has a singular arc if and only if it is at the origin, which will imply that the control is discontinuous, and that the junction has order 2. Thus, by the theorem \ref{t:q+r impar} there will be Fuller Phenomenon at this junction.
    
    The techniques from the previous section will be applied to the vector fields that define the system:
    \begin{align*}
      f(x_0,x,v)&=\left[\begin{array}{c}
	      c(x) \\
	      Tv \\
	      P(x)
	    \end{array}\right] &
      g_i(x_0,x)&=\left[\begin{array}{c}
	      0 \\
	      0 \\
	      Me_i
	    \end{array}\right].
    \end{align*}
    
    The Hamiltonian from \PMP{} is:
    \begin{equation*}
      H(x_0,x,v,p_0,p_1,p_2,u) = -p_0c(x) + \innerprod{p_1}{Tv} + \innerprod{p_2}{P(x)} - \innerprod{p_2}{Mu}.
    \end{equation*}

    Then, evaluating the Lie brackets one gets:
    \begin{align*}
      \lieprod{f}{g_i} &=  -\left[ \begin{array}{ccc}
                            0 & \frac{\partial c}{\partial x} & 0 \\
                            0 & 0 & T \\
                            0 & \frac{\partial P}{\partial x} & 0
                          \end{array} \right] \left[ \begin{array}{c}
                            0 \\
                            0 \\
                            TMe_i
                          \end{array} \right] = \left[ \begin{array}{c}
                            0 \\
                            TMe_i \\
                            0
                          \end{array} \right] &
      \lieprod{g_j}{g_i} &= \left[ \begin{array}{c}
                            0 \\
                            0 \\
                            0
                          \end{array} \right].
    \end{align*}
    
    Note that just from these Lie brackets and the vector fields $f$ and $g_i$ all the information about the problem is known. Indeed, the distribuition $\{f,g_i,\ad{f}{}{g_i}\;|\;i=1,\dots,n\}$ has dimension greater than or equal $2n$ and it is a subset of $(p_0,p_1,p_2)^\bot$. If the dimension is $2n+1$, then the singular arc is not optimal, since $(p_0,p_1,p_2)$ would be null, which contradicts \PMP{}. 
    
    On the other hand, the dimension is $2n$ if and only if $f$ is linear combination of $g_i,\ad{f}{}{g_i}$, $i=1,\dots,n$, because these fields are always linearly independent. Then, since the first entry of these fields are zero, then the first entry of $f$ is also zero. Therefore $c$ is null on a singular interval, which implies that an arc is singular if and only if $x=0$. Thus, on a singular arc $0=\dot{x}=Tv$, thus, because $T$ is invertible, $v=0$. In the same way, $0=\dot{p}_2=Tp_1$, therefore, on a singular arc $p_1$ is also null and $p_0=1$.
    
    This analysis leads to a basis of $p^\bot$, therefore  we have
    \begin{equation*}
      \Delta=\spanby{f,g_i,\ad{f}{}{g_i}\;|\;i=1,\dots,n }.
    \end{equation*}
    
    The remaining Lie brackets, although simple, have long expression that do not add any information to the analysis, then it is enough to know that if $h(x_0,x,v)$ is a vector field that do not depends on $x_0$ then $\lieprod{f}{h}$ has the expression:
    \begin{equation*}
      \lieprod{f}{h} =\left[ \begin{array}{ccc}
			0 & * & * \\
			0 & * & * \\
			0 & * & *
		      \end{array} \right] \left[ \begin{array}{c}
			-c \\
			Tv \\
			P
		      \end{array} \right] - \left[ \begin{array}{ccc}
			0 & \frac{\partial c}{\partial x} & 0 \\
			0 & 0 & T \\
			0 & \frac{\partial P}{\partial x} & 0
		      \end{array} \right] \left[ \begin{array}{c}
			* \\
			* \\
			*
		      \end{array} \right] = \left[ \begin{array}{c}
			*Tv + *P - \frac{\partial c}{\partial x}* \\
			* \\
			* \\
		      \end{array} \right]
    \end{equation*}
    where the symbol ``*'' represents arbitrary matrices of convenient dimensions that depends on $(x,v)$.
    
    Since all the vector fields do not depend on $x_0$, this expression can be applied to then and, since along a singular arc $*Tv + *P - \frac{\partial c}{\partial x}* = 0$, then $\innerprod{p}{\ad{f}{l}{g_i}} \in \Delta$ on a singular arc. Moreover, from straight computation, $B=-MT\frac{\partial^2c}{\partial x^2}TM$, which is a nonsingular matrix.
    
    Therefore, the singular arc has order 2. note that $(-1)^2B$ is a negative definite matrix, which is in accordance with the GLC Condition.
    
    Finally, from corollary \ref{c:fuller_u_nao_analitico_por_partes alpha=0} one knows that the junction is not analytic.
    
    \begin{remark}
      Even if $T$ and $M$ depend on $x$, one still has a basis of $p^\bot$, as soon as these matrix would be invertible for all values of $x$. However, it could not be true that the problem order, and thus the junction order, is even and that the $GLCS$ holds.
    
      Nevertheless, the evaluation of the problem order is trivial and one cloud still find, at least partially, a basis of $p^\bot$, so the determination of the presence of the Fuller Phenomenon could be given by the fact of $B$ to be $\Delta$-inverse-decidable. 
    \end{remark}

    \begin{remark} \label{o:hamiltoniano totalmente singular}
      The hypothesis $c(0)=0$ is fundamental, otherwise, the dimension of $\Delta$ will be always $n+1$ and $p$ will be identically null, which would imply that a singular could not be not optimal.
      
      So, Hamiltonian control systems can not have a fully singular arc, i.e., $Mp_2=0$, if the functional to be optimized do not vanish at the point that need to be reached. A particular case are the time minimum problems, which can not have a singular arc since $c(x)\equiv 1$ and $\dim{\Delta}=n+1$. This particular fact was already known \cite{CHYBA-LEONARD-SONTAG:2003}, but the now one has a more general assertion with a simple geometric interpretation.
    \end{remark}
    
    \begin{remark} \label{o:gDelta}
      Sometimes it would be possible to conclude that a vector field is not orthogonal to $p$. This can also help in the analysis of the inversibility of $B$.
      
      For instance, if it is know that a nonsingular control $u_i$ is continuous, then it is also known that $\innerprod{p}{Me_i}\neq0$. In this case, if $h \in \Delta_i \smallsetminus \Delta$, with $\Delta_i=\spanby{Me_i,\Delta}$, then $\innerprod{p}{h}\neq 0$.
    \end{remark}

\bibliography{mybib}
\bibliographystyle{abnt-num}
\nocite{ISIDORI:1989}
\nocite{KUPKA:1986}
\nocite{NIJMEIJER:1990}
\nocite{POWERS:1980}
\nocite{ZABCZYK:1992}
\nocite{MACFARLANE:2000}
\nocite{BELL:1993}
\nocite{BORISOV:2000}
\nocite{ODIA-BELL:2012}
\nocite{MEESOMBOON-BELL:2002}

\end{document}